\documentclass[a4paper,10pt]{amsart}
\usepackage[english]{babel}
\usepackage{amsmath,amssymb,amsthm,amscd}
\usepackage[all]{xy}
\newcommand{\Z}{\mathbb{Z}}
\newcommand{\ZZ}{\widehat{\mathbb{Z}}}
\newcommand{\Q}{\mathbb{Q}}
\newcommand{\PR}{\mathbb{P}}
\newcommand{\N}{\mathbb{N}}
\newcommand{\I}{\mathcal{I}}

\DeclareMathOperator{\Img}{Im}

\DeclareMathOperator{\ZLa}{ZLa}

\theoremstyle{plain}
\newtheorem{thm}{Theorem}[section]
\newtheorem{prop}[thm]{Proposition}
\newtheorem{lem}[thm]{Lemma}
\newtheorem{cor}[thm]{Corollary}
\theoremstyle{definition}

\theoremstyle{remark}
\newtheorem*{rem}{Remark}

\begin{document}
\title{Controlling distribution of prime sequences in discretely ordered principal ideal subrings of $\Q[x]$}
\renewcommand{\shorttitle}{Prime sequences in principal ideal subrings of $\mathbb{Q}[x]$}

\author{\textsc{Jana~Glivick\' a}}
\address{\textsc{Jana~Glivick\' a}: Charles University, Faculty of Mathematics and Physics, Department of Theoretical Computer Science and Mathematical Logic \\ 
Malostransk\'e n\'am\v est\'\i\ 12, 118 00 Praha~1, Czech Republic\vspace{-4pt}}
\address{and Prague University of Economics and Business, Department of Mathematics \\
Ekonomick\'a~957, 148 00 Praha~4, Czech Republic}
\email{jana.glivicka@gmail.com}

\author{\textsc{Ester~Sgallov\' a}}
\address{\textsc{Ester~Sgallov\' a}: Charles University, Faculty of Mathematics and Physics, Department of Algebra \\ 
Sokolovsk\'{a} 83, 186 75 Praha~8, Czech Republic}
\email{ester.sgallova@centrum.cz}

\author{\textsc{Jan~\v Saroch}}
\address{\textsc{Jan~\v Saroch}: Charles University, Faculty of Mathematics and Physics, Department of Algebra \\ 
Sokolovsk\'{a} 83, 186 75 Praha~8, Czech Republic}
\email{saroch@karlin.mff.cuni.cz}


\thanks{First author utilizes a long-term institutional support of research activities by Faculty of Informatics and Statistics, Prague University of Economics and Business. Second author is supported by Charles University grants PRIMUS/20/SCI/002, GA UK No.\ 742120, by the Charles University Research Center program UNCE/SCI/022 and by GA\v CR~21-00420M. Third author is supported by the grant GA\v CR~20-13778S}


\begin{abstract} We show how to construct discretely ordered principal ideal subrings of $\Q[x]$ with various types of prime behaviour. Given any set $\mathcal D$ consisting of finite strictly increasing sequences $(d_1,d_2,\dots, d_l)$ of positive integers such that, for each prime integer $p$, the set $\{p\Z, d_1+p\Z,\dots, d_l+p\Z\}$ does not contain all the cosets modulo $p$, we can stipulate to have, for each $(d_1,\dots, d_l)\in \mathcal D$, a~cofinal set of progressions $(f, f+d_1, \dots, f+d_l)$ of prime elements in our principal ideal domain $R_\tau$. Moreover, we can simultaneously guarantee that each positive prime $g\in R_\tau\setminus\N$ is either in a~prescribed progression as above or there is no other prime $h$ in~$R_\tau$ such that $g-h\in\Z$.

Finally, all the principal ideal domains we thus construct are non-Euclidean and isomorphic to subrings of the ring $\ZZ$ of profinite integers.
\end{abstract}

\maketitle
\vspace{4ex}



The study of various types of finite sequences in prime numbers is a firm and celebrated part of the field of number theory. The twin prime conjecture or the Green--Tao theorem, \cite{GT}, are just two of many universally known problems in the area. Given a strictly increasing sequence $d = (d_1,d_2,\dots, d_l)$ of positive integers such that $a, a+d_1, a+d_2,\dots, a+d_l$ are prime numbers for infinitely many $a\in\PR$, it easily follows that, for each prime number $p$, the set $\{p\Z, d_1+p\Z,d_2+p\Z,\dots,d_l+p\Z\}$ does not contain all the cosets modulo $p$. This is a simple necessary condition for the existence of cofinal sequence of (finite) prime progressions with differences prescribed by $d$.

In \cite{GS}, the authors defined and studied certain discretely ordered subrings $R_\tau$ of $\Q[x]$ where $\Z[x]\subseteq R_\tau$ and $\tau$ is a fixed parameter from the ring $\ZZ$ of profinite integers (see the preliminaries section for definitions). Among other things, they showed that all these subrings are quasi-Euclidean (also known as $\omega$-stage Euclidean, or admitting a weak Euclidean algorithm, cf.\ \cite{CZZ}) but none is actually Euclidean, see \cite[Theorems~3.2 and~3.9]{GS}. Somewhat more surprisingly, they also proved that many of these subrings are principal ideal domains. In particular, there are continuum many pairwise non-isomorphic principal ideal domains among the subrings of $\Q[x]$ according to \cite[Proposition~3.4]{GS}. This result has been further generalized in an interesting recent paper \cite{P} by Peruginelli where, among other things, a~precise description of all the PIDs $R$ with finite residue fields and $\Z[x]\subseteq R\subseteq\Q[x]$ is given in \cite[Corollary~2.18]{P}.

The aim of this paper is to have a closer look at the discretely ordered PIDs from \cite{GS} with particular emphasis on the behaviour of finite progressions of primes. We demonstrate a way how to construct a profinite integer $\tau$ such that primes are rather sparse in the PID $R_\tau$, i.e.\ for each two distinct primes $f,g\in R_\tau\setminus \Z$ the difference $f-g$ does not belong to $\Z$. On the other hand, our main result, Theorem~\ref{t:main}, shows in particular that there is a $\tau\in\ZZ$ such that $R_\tau$ is a~PID and the above mentioned simple necessary condition for the existence of a cofinal sequence of prime progressions in $R_\tau$ is actually sufficient for each $d$. All our results are rather self-contained and elementary and should be therefore accessible to a wide mathematical audience.

The structure of the paper is straightforward. After the preliminaries section where we recall basic properties of the rings $R_\tau$, we follow with a couple of elementary non-UFD examples. In Section~\ref{sec:sparse}, we present a~construction of a~PID $R_\tau$ with a~sparse set of primes, Theorem~\ref{t:sparse}. Finally, in Section~\ref{sec:main}, following two auxiliary lemmas, we prove our main result Theorem~\ref{t:main}. In the appendix, we mention some interesting logical connections regarding the rings $R_\tau$.





\section{Preliminaries and basic properties of the rings $R_\tau$}
\label{sec:prelim}

All rings appearing in this paper are associative, commutative and unital. For a~ring $R$, we denote by $R^\times$ the group of units of $R$. We denote by $\mathbb P$ the set of all primes in $\mathbb N = \{1,2,\dots\}$. For each $p\in\mathbb P$, $\mathbb Z_p$ stands for the ring of $p$-adic integers and $\Q_p$ for its field of fractions. Note that $\Z_p$ is a~discrete valuation ring with the maximal ideal $p\Z_p$ whence $\Z_p^\times = \Z_p\setminus p\Z_p$.

Further, we denote by $\mathcal A_f(\Q)$ the ring of finite adeles of $\Q$ and by $\widehat{\Z}$ its subring $\prod_{p\in\PR}\Z_p$ of profinite integers. For each $p\in\PR$, the canonical embedding of $\Q$ into $\Q_p$ induces the usual diagonal embedding of $\Q$ into $\prod_{p\in\PR}\Q_p$ which has its image in $\mathcal A_f(\Q)$. This allows us, in particular, to consider, for each $\tau = (\tau_p)_{p\in\PR}\in\ZZ$, the substitution homomorphism $d_\tau\colon\Q[x]\to\mathcal A_f(\Q)$ where $d_\tau(f) = f(\tau) = (f(\tau_p))_{p\in\PR}$ for each $f\in\Q[x]$.

Given $\tau\in\ZZ$, the ring $R_\tau$ is defined via the pull-back

\[\xymatrix{\Q[x] \ar[r]^-{d_\tau} & \mathcal A_f(\Q) \\
R_\tau \ar[u]^{\subseteq} \ar[r] & \;\ZZ, \ar[u]^{\subseteq}
}\]
equivalently $R_\tau = \{f\in \Q[x]\mid f(\tau)\in \ZZ\} = \Img(d_\tau)\cap \ZZ$. Apart from \cite{GS}, these rings appeared also in \cite[Section~6]{CP} (and in a~recent paper~\cite{P}) as special cases of polynomial overrings of integer-valued polynomials over $\Z$. In the notation therein, $R_\tau = \mathrm{Int}_\Q(\{\tau\},\ZZ)$.

\begin{rem} Bar the next section, all our choices of $\tau\in\ZZ$ are going to be transcendental over $\Q$. As a~result, the substitution homomorphism $d_\tau$ is one-one. It follows that $R_\tau$ is isomorphic to the subring $d_\tau(R_\tau)$ of $\ZZ$.
\end{rem}

Since $\tau\in\ZZ$, we have $x\in R_\tau$ and consequently $\Z[x]\subseteq R_\tau$. It follows from the definition that $R_\tau\cap\Q = \Z$ for each $\tau\in\ZZ$. This immediately yields that the canonical linear order on $\Q[x]$, where $f<g$ if and only if $g-f$ has positive leading coefficient, restricts to a~discrete order on $R_\tau$. It also implies that a $p\in\N$ is a~prime in $R_\tau$ if and only if $p\in\PR$. The primes in $R_\tau$ which are non-constant polynomials will be called \emph{non-standard}.

Finally, by \cite[Theorem~3.2]{GS}, $R_\tau$ is a quasi-Euclidean domain for each $\tau\in\ZZ$. The important implication of this fact is that $R_\tau$ is a \emph{B\'ezout domain}, i.e.\ all finitely generated ideals in $R_\tau$ are principal. Consequently, $R_\tau$ is a GCD-domain (in particular, irreducible elements are prime) and, moreover, it is a~UFD if and only if it is a~PID. Let us sum up the basic properties of the rings $R_\tau$ in the following proposition whose part $(3)$ is a mere reformulation of \cite[Lemma~3.3]{GS}.

\begin{prop}\label{p:basic} \cite{GS} Let $\tau = (\tau_p)_{p\in\PR}\in\ZZ$. Then the following hold for the subring $R_\tau$ of $\Q[x]$:
\begin{enumerate}
	\item $R_\tau$ is a B\'ezout domain and $\Z[x]\subset R_\tau \subset\Q[x]$;
	\item $R_\tau$ is discretely ordered and $R_\tau\cap\Q = \Z$, in particular each $p\in\PR$ is a~prime in $R_\tau$;
	\item $R_\tau$ is a principal ideal domain if and only if for each nonzero $h\in\Z[x]$
	\begin{itemize}
		\item $h(\tau_p)\neq 0$ for each $p\in\PR$; and
		\item $h(\tau_p)\in\Z_p^\times$ for all but finitely many $p\in\PR$.
	\end{itemize}
  \item The non-standard primes in $R_\tau$ are precisely the irreducible $f\in\Q[x]$ such that $f(\tau)\in\ZZ^\times$.
	\item If $R_\tau$ is PID, then the set of primes in $R_\tau$ is cofinal. Moreover, the non-standard primes in $R_\tau$ are exactly the polynomials $f/k\in\Q[x]$ where $f\in\Z[x]$ is non-constant, irreducible over $\Z$ and $k = \prod_{p\in\PR} p^{e_p}\in\N$ where, for each $p\in\PR$, $e_p\geq 0$ is such that $f(\tau_p)\in p^{e_p}\Z_p\setminus p^{e_p+1}\Z_p$.
\end{enumerate}
\end{prop}

\begin{proof} For $(4)$, notice that, given an $f\in R_\tau$, $f(\tau)\in\ZZ^\times$ if and only if no $p\in\PR$ divides $f$ in $R_\tau$. Thus an irreducible $f\in\Q[x]$ with $f(\tau)\in\ZZ^\times$ is necessarily prime in $R_\tau$. Assume conversely that $f$ is a~non-standard prime in $R_\tau$. Then no $p\in\PR$ divides $f$ and so $f(\tau)\in\ZZ^\times$. Suppose that $f$ is not irreducible over $\Q$. Then there exists $m\in\N$ such that $mf\in\Z[x]$ decomposes as $gh$ for some non-constant $g,h\in\Z[x]\subseteq R_\tau$. Since $f$ is prime in $R_\tau$, it follows that $f\mid g$ or $f\mid h$ holds in $R_\tau$ which is absurd since $\deg f > \deg g, \deg h$.

It remains to comment on $(5)$. In fact, it is enough to show the moreover clause since there exists an irreducible $f\in\Z[x]$ of degree $n$ (and with positive leading coefficient) for each $n\in\N$.

Given a non-constant irreducible $f\in\Z[x]$, we know from $(3)$ that the elements $e_p$ and $k$ are correctly defined. It also follows immediately from $(4)$ that $f/k\in R_\tau$ is a~prime (since $f$ is irreducible also over $\Q$).

On the other hand, if $g$ is a non-standard prime in $R_\tau$, then $g = f/m$ for a~primitive non-constant $f\in\Z[x]$ and $m\in\N$. We know that $g$ is irreducible over~$\Q$ by $(4)$ whence $f$ is irreducible over $\Z$. Write $m = \prod_{p\in\PR} p^{c_p}$ for $c_p\geq 0$. From $g\in R_\tau$, we have $f(\tau_p) \in p^{c_p}\Z_p$ and we easily see that $c_p = e_p$ and $k = m$ where the numbers $e_p$ and $k$ are defined for $f$ as in the statement (if $c_p<e_p$, then $g/p\in R_\tau$ and $g = p\cdot g/p$ would be a nontrivial decomposition of $g$ in~$R_\tau$).
\end{proof}

\section{Straightforward non-UFD examples}
\label{sec:examples}

Let us first examine the case $\tau\in\Z$. By the definition, $R_0 = x\Q[x]+\Z$, a~notoriously known GCD-domain which is not a UFD. Moreover, for each $z\in\Z$, the ring automorphism $\iota_z\colon\Q[x]\to \Q[x]$ taking $x$ to $x-z$ restricts to an isomorphism of $R_0$ and $R_z$. It is easy to see that the irreducible (equivalently prime) elements in $R_0$ are precisely the elements $\pm p$, where $p\in\PR$, and the non-constant polynomials $f\in\Q[x]$ which are irreducible over $\Q$ and satisfy $f(0)\in\{-1,1\}$. In particular, the polynomials $x^n/2-1$ and $x^n/2+1$ where $n$ runs through $\N$ form a cofinal set in $R_0$ consisting of twin primes.

\smallskip

Next, we give a slightly less trivial example of a ring $R_\tau$ such that the set of primes in $R_\tau$ is just $\PR$. The $\tau = (\tau_p)_{p\in\PR}\in\ZZ$ is going to be chosen in such a way that, for each $p\in\PR$, the $\tau_p$ is actually an integer. First, let us recall a well-known

\begin{lem}\label{l:w-k} Let $f\in\Z[x]$ be a non-constant polynomial. Then the set \[S_f = \{p\in\PR\mid (\exists z\in\Z)\, p\mid f(z)\} = \{p\in\PR\mid \Z/p\Z\models (\exists v)\, f(v) = 0\}\] is infinite. Consequently, we can fix for each non-constant $f\in\Z[x]$ an infinite subset $T_f$ of $S_f$ in such a way that $T_f\cap T_g = \varnothing$ whenever $f\neq g$.
\end{lem}

\begin{proof} Let $f = \sum_{i = 0}^n a_ix^i$, $n\in\N$, $a_n\neq 0$. The lemma trivially holds if $a_0 = 0$ (or, more generally, if $f$ has a root in $\Z$). Let us assume it is not the case. Then $f(a_0x) = a_0(a_na_0^{n-1}x^n+\dots + a_1x + 1)$ whence we can w.l.o.g.\ assume that $a_0 = 1$.

Assume, for the sake of contradiction, that $S_f = \{p_1,\dots, p_m\}$ for some $m\in\N$. Note that $S_f$ has to be nonempty since $f$ is non-constant. Let $p = \prod_{i = 1}^m p_i$. Take $k\in\N$ large enough so that $f(p^k) \notin \{-1,0,1\}$. Then $f(p^k)$ is not divisible by any prime $p_1,\dots, p_m$ since $a_0 = 1$. But it has to be divisible by some prime which should have belonged to $S_f$ by its very definition; a contradiction.

The last part easily follows from the fact that all the sets $S_f$ as well as $\Z[x]$ are countably infinite.
\end{proof}

\begin{prop}\label{p:justprimes} There exists $\tau\in\ZZ$ such that the set of primes in $R_\tau$ is just $\PR$. Moreover, each element from $R_\tau\setminus\Z$ possesses infinitely many prime divisors.
\end{prop}

\begin{proof} Assume that we have fixed infinite sets $T_f$ of primes as in Lemma~\ref{l:w-k}. Let $p\in\PR$. If $p$ does not belong to any set $T_f$, put $\tau_p = 0$ (in fact, $\tau_p$ can be chosen arbitrarily from $\Z_p$). Otherwise, there is precisely one non-constant $f\in \Z[x]$ such that $p\in T_f$, and we put $\tau_p = z$ for a suitable $z\in\Z$ such that $p\mid f(z)$ in $\Z$. Set $\tau = (\tau_p)_{p\in\PR}$.

Assume that $g\in R_\tau\setminus\Z$. Then $g = f/k$ for some non-constant $f\in\Z[x]$ and $k\in\N$. It follows that, for any $p\in T_f$ with $p\nmid k$, we have $p\mid f(\tau_p)$ in $\Z$ and hence also $p\mid g$ in $R_\tau$. There is, indeed, infinitely many such primes $p$.
\end{proof}

\section{PIDs with sparse set of primes}
\label{sec:sparse}

In what follows, we denote by $\I$ the set of all non-constant irreducible (over $\Z$) polynomials from $\Z[x]$ with positive leading coefficient. In particular, if $f,g\in\I$ are distinct, then they possess no common root in $\mathbb C$. Furthermore, $f+p\Z[x]\neq 0$ for each $f\in\I$ and $p\in\PR$ since the content of every polynomial from $\I$ is $1$.

The strategy in the proof of the next theorem is to define the components $\tau_p\in\Z_p$ by a recursive process for larger and larger primes so that, in the $i$th step, we have $\tau_p$ defined at least for all primes $\leq s(i)$ where $s\colon\N\to \N$ is going to be a strictly increasing function. In order to get the intended sparseness of primes, we will occasionally need to define $\tau_p$ in the $i$th step for some $p > s(i)$ as well, however, in each step of this recursive process, it will be the case that $\tau_p$ is defined only for finitely many primes $p$. Moreover, simultaneously with the definition of $\tau_p$, we are going to gradually fix the positive non-standard primes in the resulting ring $R_\tau$.

\begin{thm}\label{t:sparse} There exists $\tau\in\ZZ$ such that $R_\tau$ is a PID and, for each prime $g\in R_\tau\setminus\Z$, no element of the form $g+z$ where $z\in \Z$ is a prime.
\end{thm}

\begin{proof} Let us enumerate $\mathcal I = \{f_1, f_2,\dots\}$ and define a function $s\colon\N \to \N$ by the assignment \[s(i) = \sum_{j = 1}^i \deg(f_j).\] By a~recursive process, we are going to define a desired $\tau = (\tau_p)_{p\in\PR}\in\ZZ$. For each prime $p\leq s(1)$, we choose $\tau_p\in\Z_p$ arbitrarily in such a~way that $\tau_p$ is not a~root of any nonzero polynomial from $\Z[x]$. This can be done since the set of all roots of nonzero polynomials from $\Z[x]$ is countable whilst $\Z_p$ is uncountable.

At the same time, we fix $n_1 = \prod_{p\in\PR, p\leq s(1)}p^{e_p}\in\N$ with $e_p\geq 0$ in such a way that $f_1(\tau_p)\in p^{e_p}\Z_p\setminus p^{e_p+1}\Z_p$ holds for each prime $p\leq s(1)$. Put $g_1 = f_1/n_1$. It follows, that $g_1(\tau_p)$ is a unit in $\Z_p$ for each prime $p\leq s(1)$.

Assume now that $i>1$ and $\tau_p$ is defined (at least) for each prime $p\leq s(i-1)$. Assume also that $n_1, \dots, n_{i-1}\in\N$ are defined in such a way that, for $g_j = f_j/n_j\in\Q[x]$, the $g_j(\tau_p)$ is a unit in $\Z_p$ for each $j\in\{1,\dots, i-1\}$ and each $p\in\PR$ for which $\tau_p$ is already defined. Moreover, assume that, for each distinct $j,k\in\{1,\dots, i-1\}$, the difference $g_j-g_k\in\Q[x]$ is not an integer (more precisely, a~constant polynomial).

For each prime $p\leq s(i)$ for which $\tau_p$ is not defined (in particular, for which $s(i-1)<p$), we choose $\tau_p\in\Z_p$ so that

\begin{itemize}
	\item $f_j(\tau_p)$, equivalently $g_j(\tau_p)$, is a unit in $\Z_p$ for each $j \in\{1,\dots, i-1\}$, and
	\item $\tau_p$ is not a root of any nonzero polynomial from $\Z[x]$.
\end{itemize}

The first part merely states that $c:=\tau_p+p\Z_p$ is not a root of any $f_1,\dots, f_{i-1}$ in the $p$-element field $\Z_p/p\Z_p$. We can pick such a $c\in\Z_p/p\Z_p$ since $f_j+p\Z[x]\neq 0$ for each $j = 1,\dots, i-1$ and $p>s(i-1) = \sum_{j = 1}^{i-1}\deg(f_j)$. The set of all $\tau_p\in\Z_p$ such that $c = \tau_p+p\Z_p$ is uncountable hence we can pick from it a~suitable $\tau_p$ satisfying also the second condition above.

Let $S$ be the (finite) set of all primes $p$ for which $\tau_p$ is already defined. We put $n_i^\prime = \prod_{p\in S} p^{e_p}$ where $e_p\geq 0$ and $f_i(\tau_p)\in p^{e_p}\Z_p\setminus p^{e_p+1}\Z_p$. If there is no $m\in\{1,\dots, i-1\}$ such that $g_m - f_i/n_i^\prime\in\Z$, we put $n_i = n_i^\prime$ and skip the rest of this paragraph. Otherwise, we fix a $k\in\N$ large enough so that $g_m-f_i/(tn_i^\prime)\notin\Z$ holds for all $m\in\{1,\dots, i-1\}$ whenever $t\in\N$ is greater than $k$.
Now we choose a~$q\in\PR$ such that $q>k$, $\tau_q$ is not yet defined and with the property that \[\Z/q\Z\models (\exists v)\, f_i(v) = 0\,\wedge\, \prod_{j = 0}^{i-1} f_j(v) \neq 0.\eqno{(*)}\] Assume for a moment that it is not possible. By Lemma~\ref{l:w-k}, we know that there is infinitely many primes $q>k$ for which $\tau_q$ is not yet defined and such that $\Z/q\Z\models (\exists v)\, f_i(v)=0$. It follows that $(\exists v)\, f_i(v) = 0\,\wedge\, \prod_{j = 0}^{i-1} f_j(v) = 0$ holds in the field $\Z/q\Z$, whence also in its algebraic closure, for all these primes $q$. Employing the Lefschetz principle, \cite[Corollary~2.2.10]{M}, we deduce that the polynomials $f_i$ and $\prod_{j = 1}^{i-1} f_j$ share a~common root in $\mathbb C$, a contradiction. Thus we have a desired $q$ for which $(*)$ holds and we can thus fix a $c\in \Z_q/q\Z_q$ such that $f_i(c) = 0$ and $\prod_{j = 1}^{i-1} f_j(c) \neq 0$ in this $q$-element field. Consequently, we pick a~$\tau_q\in\Z_q$ in such a~way that $c = \tau_q+q\Z_q$ and $\tau_q$ is not a root of any nonzero polynomial from $\Z[x]$. Let $e_q>0$ be such that $f_i(\tau_q)\in q^{e_q}\Z_q\setminus q^{e_q+1}\Z_q$. Set $n_i = q^{e_q}n_i^\prime$ and $g_i = f_i/n_i$. Since $q^{e_q}>k$, we know that $g_m-g_i\notin\Z$ now holds for all $m\in\{1,\dots, i-1\}$.

By the choice of $n_j$, we see that $g_j(\tau_p)$ is a unit in $\Z_p$ for each $j\in\{1,\dots, i\}$ and each $p\in\PR$ for which $\tau_p$ is currently defined.
We can proceed to the step $i+1$.

\smallskip

Having finished the definition of $\tau = (\tau_p)_{p\in\PR}$, it is now easy to check that $R_\tau$ has the desired properties. First of all, $R_\tau$ is a PID by Proposition~\ref{p:basic} (3): for a~nonzero $h\in\Z[x]$ we made sure in the construction above that $h(\tau_p)\neq 0$ for every $p\in\PR$. Moreover, we also guaranteed that, for each $i\in\N$, $f_i(\tau_p)$ is a unit in $\Z_p$ for all but finitely many primes $p\in\PR$; note that the set of these finitely many primes is precisely the set of prime divisors of $n_i$ in $\Z$. It thus follows that also $h(\tau_p)$ is a unit in $\Z_p$ for all but finitely many primes $p\in\PR$ since $h = m\prod_{j=1}^n f_{i_j}$ for some nonzero $m\in\Z$, $n\geq 0$ and $i_1,\dots, i_n\in\N$.

So $R_\tau$ is a PID and Proposition~\ref{p:basic}~(5) yields that the set of positive primes in $R_\tau$ is exactly $\PR\cup\{g_i\mid i\in\N\}$. Finally, it follows immediately from the construction that the difference of two distinct non-standard primes in $R_\tau$ is never an integer.
\end{proof}

\section{PIDs with cofinal set of twin primes\dots and much more}
\label{sec:main}

In what follows, we keep the notation $\mathcal I$ from the previous section. Our aim in this section is to combine the approach from the proof of Theorem~\ref{t:sparse}, which allowed us to make primes in $R_\tau$ sparse, with a construction going in the opposite direction---ensuring that there is a cofinal sequence of finite progressions of primes in~$R_\tau$ with prescribed differences.

We start by proving two short lemmas. The first one provides viable candidates for the prime progressions (depending on the input data $F$, $P$, $n$ and $d$). The primes in any progression are going to be elements from $\mathcal I$ specified in part $(1)$ of Lemma~\ref{l:largeprimes}. The second lemma then establishes what are the suitable choices of the parameter $k$.

\begin{lem}\label{l:largeprimes} Let $F$ be a finite subset of $\mathcal I$, $P\subseteq\PR$ infinite, $n>1$ an integer and $d = (d_1,d_2,\dots, d_l)$ a strictly increasing sequence of positive integers. Put $d_0 = 0$. Then there exist $r,a\in\N$ and pairwise distinct $p_0,p_1,\dots, p_l\in P$ such that the following hold for each non-negative integer $k$ and $p := \prod_{i = 0}^l p_i$:
\begin{enumerate}
  \item $x^n + kpx + a + d_i\in\mathcal I$ for each $0\leq i\leq l$;
	\item $r^n + kpr + a + d_i$ and $p$ are coprime for each $0\leq i\leq l$;
	\item $f(r)$ and $p$ are coprime for each $f\in F$.
\end{enumerate}

\end{lem}

\begin{proof} Let $r\in\N$ be arbitrary such that $r^n>d_l$ and $r$ is not a root of the polynomial $g:=\prod_{f\in F}f$. It follows that $q\mid g(r)$ holds only for finitely many primes $q\in\PR$. Possibly omitting such primes from the set $P$, we can thus w.l.o.g.\ assume that $q\nmid g(r)$ holds for each $q\in P$.

Pick $p_0,\dots, p_l\in P$ pairwise distinct and each strictly greater than $r^n+d_l$. By the paragraph above, we know that $(3)$ holds true. Now let $a\in\N$ be a~solution of the system of congruences $y\equiv p_i-d_i \pmod {p_i^2}$ where $i$ runs through $\{0,1,\dots, l\}$. It follows that $p_i\mid a+d_i$ and $p_i^2\nmid a+d_i$ for each $i\in \{0,1,\dots, l\}$. Using the Eisenstein's criterion over $\Z$, we obtain $(1)$.

Assuming for a while that $p_j \mid r^n + kpr + a + d_i$ for some $0\leq i,j\leq l$, we use $p_j\mid kpr+a+d_j$ to deduce that $p_j\mid r^n + d_i-d_j$ which leads to a~contradiction since we have $p_j> r^n + d_l \geq r^n + d_i-d_j\geq r^n-d_l>0$. Thus $(2)$ holds true as well.
\end{proof}

In the rest of the paper, let us denote by $\mathcal S$ the set of all nonempty finite strictly increasing sequences $d = (d_1,d_2,\dots, d_{l_d})$ of positive integers satisfying, for each $p\in\PR$, that $\{p\Z,d_1+p\Z,\dots, d_{l_d}+p\Z\}\subsetneq\Z/p\Z$. In particular, for each $d\in\mathcal S$, all the components of $d$ are even numbers.

\begin{lem} \label{l:manyk} Let $Q$ be a finite subset of $\,\PR$ and $p\in\N$ be such that $q\nmid p$ for each $q\in Q$. Let a~$c_q\in \{1,\dots,q-1\}$ be fixed for each $q\in Q$. Assume that $d = (d_1,\dots, d_l)\in\mathcal S$ and $a,n\in\Z$ with $n>1$ are given. Put $d_0 = 0$.

Then there exist infinitely many $k\in\N$ such that $q\nmid c_q^n+kpc_q+a+d_i$ for every $q\in Q$ and $i\in\{0,\dots, l\}$.
\end{lem}

\begin{proof} Fix an arbitrary $q\in Q$. By our assumption $q\nmid pc_q$. Since $d_0, d_1, \dots, d_l$, by the definition of $\mathcal S$, do not cover all residue classes modulo $q$, we can choose $d_q\in\{0,1,\dots, q-1\}$ such that $q\nmid c_q^n+kpc_q+a+d_i$ whenever $k\equiv d_q\pmod q$ and $0\leq i\leq l$. It follows that any $k\in\N$ solving the system of congruences $y\equiv d_q\pmod q$ where $q$ runs through $Q$ is, indeed, the desired one.
\end{proof}

We are now ready to prove the main theorem of this paper. For the special choice $\mathcal D = \varnothing$, we recover Theorem~\ref{t:sparse}. While defining $\tau_p\in \Z_p$ in the proof below, we say that $\tau_p$ is \emph{admissible} if it is a~unit in $\Z_p$ and it is not a~root of any non-constant polynomial over~$\Z$. All the $\tau_p$ are going to be chosen admissible.

\begin{thm} \label{t:main} Let $\mathcal D\subseteq\mathcal S$. There exists $\tau = (\tau_p)_{p\in\PR}\in\ZZ$ such that $R_\tau$ is a~PID and

\begin{enumerate}
	\item[(i)] for each $d = (d_1,\dots,d_l)\in \mathcal D$, the discretely ordered domain $R_\tau$ contains a~cofinal sequence consisting of progressions $(f, f+d_1,\dots, f+d_l)$ of positive primes in $R_\tau$;
	\item[(ii)] each positive non-standard prime $g$ in $R_\tau$ is either contained in a progression as above, or there is no other prime $h\in R_\tau$ such that $g-h\in\Z$.
\end{enumerate}

\end{thm}

\begin{proof} Put $\mathcal E = (\mathcal D\cup \{\varnothing\})\times (\N\setminus\{1\})$ and fix a bijection $\iota\colon\N\to \mathcal E$. Let us also well order the set $\mathcal I$ with the ordinal type $\omega$. Put $\mathcal I_0 = \varnothing$ and $s_0 = 1$.

We proceed by induction. Suppose that $s_j$ and a~finite $\mathcal I_j\subset \mathcal I$ are defined for each $0\leq j<m\in\N$ in such a~way that $s_j = 1+\sum _{g\in\mathcal I_j} \deg g$ and $\mathcal I_0\subset \mathcal I_1\subset\dotsb\subset\mathcal I_{m-1}$. Suppose also that $\tau_p$ is defined at least for each prime $p\leq s_{m-1}$ (but possibly also for finitely many primes $p > s_{m-1}$). Finally, assume that each $g\in\mathcal I_{m-1}$ has a~positive integer $n_g$ assigned to it in such a~way that $\frac{g}{n_g}(\tau_p)$ is a~unit in $\Z_p$ for all $p$ for which $\tau_p$ is defined. This integer $n_g$, once assigned to $g$, does not change and $p\in\PR$ can be a~prime divisor of $n_g$ only if $\tau_p$ had already been defined at the time when $n_g$ was fixed. In particular, after the $n_g$ is fixed, all the newly defined $\tau_q$, $q\in\PR$, will have the property that $g(\tau_q)$ is a~unit in $\Z_q$ if and only if $\frac{g}{n_g}(\tau_q)$ is a~unit in $\Z_q$.

Set $(d,n) = \iota(m)$. Assume first that $d = \varnothing$. Then we pick the least (in the well order fixed above) $f\in\mathcal I\setminus\mathcal I_{m-1}$ and set $s_m = s_{m-1} + \deg f$ and $\mathcal I_m = \mathcal I_{m-1}\cup\{f\}$. For each prime $p\leq s_m$ for which $\tau_p$ has not been defined yet (in particular, $p>s_{m-1}$), we pick an admissible $\tau_p$ such that $g(\tau_p)$ is a~unit in $\Z_p$ for each $g\in \mathcal I_{m-1}$. This can be done by the property of $s_{m-1}$.

As in the proof of Theorem~\ref{t:sparse}, we let $S$ denote the (finite) set of all primes $p$ for which $\tau_p$ is already defined. We follow the respective paragraph (employing the Lefschetz principle) in the proof of Theorem~\ref{t:sparse}, i.e.\ we define at most one more $\tau_q$ and a~suitable $n_f\in\N$ such that $\frac{f}{n_f}(\tau_p)$ is a~unit in $\Z_p$ for all $p$ for which $\tau_p$ is currently defined (the possible new $q$ included), and $\frac{g}{n_g}-\frac{f}{n_f}\notin\Z$ holds for each $g\in \mathcal I_{m-1}$.

Now assume that $\varnothing\neq d = (d_1,\dots, d_l)$. Then we put $s_m = s_{m-1} + (l+1)n$; recall that $(d,n) = \iota(m)$. Again, for each prime $p\leq s_m$ for which $\tau_p$ has not been defined yet, we choose an admissible $\tau_p$ such that $g(\tau_p)$ is a~unit in $\Z_p$ for each $g\in \mathcal I_{m-1}$.

Put $F = \mathcal I_{m-1}\cup\{x\}$ and let $Q$ denote the (finite) set of all primes for which $\tau_p$ is currently defined. Set $P = \PR\setminus Q$. Use Lemma~\ref{l:largeprimes} to obtain suitable $r,a\in\N$ and pairwise distinct $p_0,p_1,\dots, p_l\in P$ for the setting $F, P, n, d$. For each $i\in\{0,\dots, l\}$, the part $(2)$ of the lemma together with the part $(3)$ for $f = x$ allows us to fix an admissible $\tau_{p_i}\in\Z_{p_i}$ with $\tau_{p_i} + p_i\Z_{p_i} = r + p_i\Z_{p_i}$.

Put $d_0 = 0$, $p = \prod_{i = 0}^l p_i$. By the paragraph above, $p_i\notin Q$ for $i\in\{0,1,\dots,l\}$. For each $q\in Q$, let $c_q\in\{1,\dots,q-1\}$ be such that $\tau_q+q\Z_q = c_q+q\Z_q$. Utilizing Lemma~\ref{l:manyk}, we choose a~sufficiently large $k\in\N$ such that $q\nmid c_q^n+kpc_q+a+d_i$ for every $q\in Q$ and $i\in\{0,\dots, l\}$, and also satisfying that $f_k - g/n_g\notin\Z$ for the polynomial $f_k = x^n+kpx+a$ and each $g\in \mathcal I_{m-1}$. In particular, $f_k + d_i \neq g$ for each $g\in \mathcal I_{m-1}$ and $0\leq i\leq l$: this is clear if $n_g = 1$, and it follows from the fact that, for any prime $q\mid n_g$, we have $q\in Q$ and $q\mid g(c_q)$ whereas $q\nmid f_k(c_q)+d_i$ for each $q\in Q$ and $0\leq i\leq l$. Using this together with part $(1)$ of Lemma~\ref{l:largeprimes}, we have thus got that $f_k+d_i\in\mathcal I\setminus\mathcal I_{m-1}$ holds for each $0\leq i\leq l$. Finally, we set $n_{f_k+d_i} = 1$ for each $i\in\{0,\dots,l\}$ and $\mathcal I_m = \mathcal I_{m-1}\cup\{f_k+d_i\mid 0\leq i\leq l\}$, and notice that $f(\tau_{p_i})$ is a~unit in $\Z_{p_i}$ for each $f\in\mathcal I_m$ and $0\leq i\leq l$ by Lemma~\ref{l:largeprimes}, parts $(2)$ and $(3)$.

This completes the induction. Since there are infinitely many pairs $(d,n)\in\mathcal E$ with $d = \varnothing$, we see that $\mathcal I = \bigcup_{m= 0}^\infty \mathcal I_m$: indeed, in the worst case scenario, we have to put the $j$th smallest element of $\mathcal I$ into $\mathcal I_m$ if $m$ is the $j$th smallest positive integer such that $\iota(m) = (\varnothing,n)$ for some $n$. Notice also that the sequence $(s_m)_{m=0}^\infty$ is strictly increasing, whence $\tau_p$ is now defined for every $p\in\PR$ and we can put $\tau = (\tau_p)_{p\in\PR}$.

In the construction, we gradually fixed admissible $\tau_p\in\Z_p$ and put aside polynomials from $\mathcal I$ into gradually expanding finite sets $\mathcal I_m$, $m\geq 0$. The important point of this process is that, once $\mathcal I_m$ was defined, all the $\tau_p$ defined afterwards satisfy that $g(\tau_p)$ is a~unit in $\Z_p$ for each $g\in\mathcal I_m$. This observation together with the admissibility of all the $\tau_p$ implies that each $g\in\mathcal I$ has in $R_\tau$ only finitely many divisors from $\PR$ and that, for each $p\in\PR$, there is $e\in\N$ such that $p^e\nmid g$ in $R_\tau$. In particular, $R_\tau$ is a~PID.

Each sequence $(f+d_i)_{i=0}^l$ of polynomials of degree $n>1$ newly added into $\mathcal I_m$ in the step when $\iota(m) = (d,n)$ with $d\neq\varnothing$ was, in addition, chosen in such a~way that $f+d_i(\tau_p)$ be a~unit in $\Z_p$ for every $p\in\PR$ whatsoever. In particular, the polynomials $f+d_i$, where $0\leq i\leq l$, are primes in $R_\tau$ and (i) holds true. The part (ii) then follows from the construction and the fact that each positive non-standard prime in $R_\tau$ is of the form $g/n_g$ for some $g\in\mathcal I$.
\end{proof}

As a special case, we get an immediate corollary.

\begin{cor}\label{c:special} There exists a $\tau\in\ZZ$ such that $R_\tau$ is a~PID and:
\begin{enumerate}
	\item $R_\tau$ contains a~cofinal set of twin primes;
	\item for each $l\in\N$, $R_\tau$ contains a~cofinal sequence consisting of arithmetic progressions of primes of length $l$.
\end{enumerate}
\end{cor}

\begin{proof} Use Theorem~\ref{t:main} for $\mathcal D = \mathcal S$. The first part then follows from $(2)\in\mathcal S$ whilst the second part utilizes that $(l!, 2l!,\dots, (l-1)l!)\in\mathcal S$ for each $l>1$.
\end{proof}

\appendix
\section{Logical connections}
\label{sec:appendix}

It follows from \cite[Section~2]{GS} that, in any weakly saturated Peano ring $S$, i.e.\ a~weakly saturated ring elementarily equivalent with $\Z$, all the rings $R_\tau$, $\tau\in\ZZ$, can be found as subrings. As a~consequence, every universal sentence in the language of rings which holds in $\Z$ (equivalently in $S$) holds in all the domains $R_\tau$ as well. If we fix a~particular embedding of $R_\tau$ into $S$, it also yields that, if $f\in R_\tau$ is a~prime in $S$, then it is a~prime in $R_\tau$. Naturally, the converse implication does not hold. It is also not true that the (non-negative parts of) rings $R_\tau$ satisfy induction for open formulas in the language of rings.

Since $\Z[x]\subseteq R_\tau$ holds for each $\tau\in\ZZ$, all the rings $R_\tau$ are, in particular, faithful $\Z[x]$-modules. More-or-less following Glivick\'y's Ph.D.\ thesis \cite{G} (see also \cite{GP}), we define a~first order theory $\ZLa$ in the language $\mathcal L = (+,-,0,1,\leq,\,\cdot f\mid f\in\Z[x])$. The axioms of $\ZLa$ are:

\begin{enumerate}
	\item all the axioms of (right) $\Z[x]$-modules;
	\item axioms expressing that $\leq$ is a linear order satisfying $0<1$, $v\leq 0\vee 1\leq v$ and $(v_1\leq v_2 \wedge w_1\leq w_2)\rightarrow v_1+w_1\leq v_2+w_2$;
  \item the axiom $0\leq v\rightarrow 0\leq v\cdot x$;
  \item the axiom scheme $1\cdot x \neq 1\cdot n$ for each non-negative integer $n$;
	\item the axiom scheme $0<1\cdot f\rightarrow (\exists w)\, w\cdot f \leq v < (w+1)\cdot f$ for each $f\in\Z[x]$ (i.e.\ Euclidean division by $1\cdot f$).
\end{enumerate}    

One of the main results of \cite{G}, Theorem~1.4:5, is a thorough description of the theory $\ZLa$. The author shows that it has elimination of quantifiers up to disjunction of positive primitive formulas. As a consequence, it is model-complete, i.e.\ whenever $A, B$ are models of $\ZLa$ and $A$ is a~submodel of $B$, then $A$ is an~elementary submodel of $B$. He also shows that the induction for all $\mathcal L$-formulas holds in (the non-negative part of) models of $\ZLa$. The most interesting result in our context is, however, the following one. 

\begin{prop} \label{p:primemodels} $($\cite[Theorem~1.4:5]{G}$)$ Consider $R_\tau$, $\tau\in\ZZ$, as $\mathcal L$-structures. Then $\{R_\tau\mid \tau\in\ZZ\}$ is a~representative set of all models of $\ZLa$ up to elementary equivalence. In fact, the $R_\tau$ are even all the unique prime models of $\ZLa$, i.e.\ for each model $M$ of $\ZLa$ there exists a~unique $\tau\in\ZZ$ such that $R_\tau$ (elementarily) embeds into $M$.
\end{prop}

For each $n\in\N$, let $A_n$ denote the abelian group $\Z$ considered as an $\mathcal L$-structure with the usual order relation $\leq$ and $\cdot x$ interpreted as multiplication by $n$. Then $A_n$ satisfies all the axioms of the theory $\ZLa$ bar the scheme $(4)$. By \cite[Corollary~1.4:6]{G}, the ultraproducts $\prod_{n\in\N}A_n/\,\mathcal U$ where $\mathcal U$ is a non-principal ultrafilter on~$\N$ also represent all the models of $\ZLa$ up to elementary equivalence.

Of course, this does not help to transfer, let us say, the result about twin primes from a~ring $R_\tau$ constructed in Corollary~\ref{c:special} to $\Z$. We have not expected that any such thing would be possible. On the other hand, with regard to the main results in~\cite{GP}, it could be interesting to have a~closer look at the behaviour of prime sequences also in the rings $R_{\tau_1,\tau_2} = \{f\in\Q[x,y]\mid f(\tau_1,\tau_2)\in\ZZ\}$ where $\tau_1,\tau_2\in\ZZ$ and $f(\tau_1, \tau_2)$ is evaluated in $\prod_{p\in\PR}\Q_p$. This might be a~much harder task though.

\bigskip

\noindent\textit{Acknowledgement.} We thank the anonymous referee for carefully reading our manuscript and for making several suggestions which helped us to further enhance the quality of the paper.


\end{document}